\documentclass[11pt]{amsart}

\usepackage[T1]{fontenc}
\usepackage{lmodern}
\usepackage{amsmath,amssymb,mathtools}
\usepackage{enumitem}
\usepackage{microtype}
\usepackage{needspace}
\usepackage[colorlinks=true,linkcolor=blue,citecolor=blue,urlcolor=blue]{hyperref}

\setlist[enumerate]{label=(\roman*),leftmargin=2.2em}
\setlist[itemize]{leftmargin=1.8em}

\newtheorem{theorem}{Theorem}[section]
\newtheorem{lemma}[theorem]{Lemma}
\newtheorem{proposition}[theorem]{Proposition}
\newtheorem{corollary}[theorem]{Corollary}

\theoremstyle{definition}

\newtheorem{remark}[theorem]{Remark}

\newcommand{\calO}{\mathcal O}

\title[The prescribed-vertex threshold for $C_{3q}$]{The Prescribed-Vertex Semidegree Threshold for Directed $3q$-Cycles in Oriented Graphs}
\author{Zhenhua Lyu}
\address{School of Science, Shenyang Aerospace University, Shenyang 110136, China}
\email{lyuzhh@outlook.com}
\date{}
\hypersetup{
  pdftitle={The Prescribed-Vertex Semidegree Threshold for Directed 3q-Cycles in Oriented Graphs},
  pdfauthor={Zhenhua Lyu}
}

\subjclass[2020]{Primary 05C20; Secondary 05C35, 05C38}

\keywords{oriented graph, prescribed vertex, directed cycle, minimum semidegree, stability}

\begin{document}

\begin{abstract}
For every $q\ge2$, we prove that every oriented graph $G$ on
$n\ge45q-8$ vertices whose minimum semidegree satisfies
\[
  \delta^0(G)\ge \left\lceil\frac n3\right\rceil
\]
contains a directed cycle of length $3q$ through every vertex.  The
semidegree bound is sharp.  This closes the one-unit gap left by the
prescribed-vertex theorem of Kelly, K\"uhn and Osthus when $3\mid n$.
We also prove that  if an oriented graph
$H$ has order $N$, minimum semidegree $d\ge3$, and $7d\ge2N+3$, then
every ordered pair of distinct vertices is joined by a path of length
three, four, or five.
The constant $+3$ is best possible.  As a consequence, the order
hypothesis $n\ge10^{10}\ell$ in the general prescribed-vertex theorem
of Kelly, K\"uhn and Osthus can be replaced by $n\ge15\ell-60$ for
$\ell\ge7$.
\end{abstract}

\maketitle

\section{Introduction}

An \emph{oriented graph} is an orientation of a simple graph.  For a vertex $v$ of an oriented graph $G$, let $d^+(v)$ and $d^-(v)$ denote its outdegree and indegree.  Denote
\[
  \delta^\pm(G)=\min_{v\in V(G)}d^\pm(v),
  \qquad
  \delta^0(G)=\min\{\delta^+(G),\delta^-(G)\}.
\]
All paths and cycles in this paper are directed unless stated otherwise.  We write $C_\ell$ for the consistently oriented cycle of length $\ell$.

Degree conditions for directed cycles in oriented graphs, and in particular
for Hamilton cycles, are surveyed by Bermond and Thomassen~\cite{BT} and by
K\"uhn and Osthus~\cite{KOsurvey}.  We consider the minimum semidegree condition
that forces a directed cycle of a given length through each specified
vertex.

For $\ell\ge3$, let $\tau^v_\ell(n)$ be the least integer $d$ such
that every $n$-vertex oriented graph $G$ with $\delta^0(G)\ge d$
contains a copy of $C_\ell$ through every vertex.  The superscript
emphasizes that this is a prescribed-vertex, or vertex-rooted, threshold:
it is stronger than merely requiring one copy of $C_\ell$ somewhere in
the graph.

Kelly, K\"uhn and Osthus~\cite{KKO} proved that, for every $\ell\ge4$
and every $n\ge10^{10}\ell$, the condition
\[
  \delta^0(G)\ge \left\lfloor\frac n3\right\rfloor+1
\]
forces a copy of $C_\ell$ through any prescribed vertex.  When
$3\mid\ell$, their modified cyclic blow-up
\cite{KKO} gives a prescribed vertex
contained in no $C_\ell$ while
\[
  \delta^0(G)=\left\lfloor\frac{n-1}{3}\right\rfloor.
\]
Consequently, for $\ell=3q\ge6$ and $n\ge10^{10}\ell$,
\begin{equation}\label{eq:old-window}
  \left\lceil\frac n3\right\rceil
  \le \tau^v_{3q}(n)
  \le \left\lfloor\frac n3\right\rfloor+1.
\end{equation}
The two bounds coincide when $n\equiv1,2\pmod3$.  For $n=3m$,
however, they leave precisely the one-unit window
\[
  m\le \tau^v_{3q}(3m)\le m+1.
\]
Denote
\begin{equation}\label{eq:n0-def}
  n_0(q)=
  \begin{cases}
    37,&if~q=2,\\
    97,&if~q=3,\\
    45q-8,&if~q\ge4.
  \end{cases}
\end{equation}
We present the main results as follows.
\begin{theorem}\label{thm:main}
Let $q\ge2$ and $n\ge n_0(q)$.  If $G$ is an oriented graph on $n$ vertices and
\[
  \delta^0(G)\ge \left\lceil\frac n3\right\rceil,
\]
then every vertex of $G$ lies on a copy of $C_{3q}$.
\end{theorem}

\begin{corollary}\label{cor:threshold}
For every $q\ge2$ and $n\ge n_0(q)$,
\[
  \tau^v_{3q}(n)=\left\lceil\frac n3\right\rceil.
\]
\end{corollary}

For $3\mid n$ and $n\ge n_0(q)$, this replaces the upper bound
$m+1$ in~\eqref{eq:old-window} by $m$.

The order bound in Theorem~\ref{thm:main} follows from a three-length
linking lemma.  With $\ell=3q$, the hypothesis $n\ge n_0(q)$ is
$n\ge15\ell-8$ for $q\ge4$, in place of $n\ge10^{10}\ell$.
If an oriented graph $H$ has order $N$ and minimum semidegree $d\ge3$,
then
\[
  7d\ge2N+3
\]
guarantees, between every ordered pair of distinct vertices, a path of
length three, four, or five.  We give examples with
$7d=2N+2$, so the additive constant is sharp.  In the defect
parameterization $d\ge N/3-C+1$, where $C$ is a positive integer, the
lemma already applies when
$N\ge21C-12$. The coefficient $21$ is asymptotically best possible.

Jackson~\cite{Jackson} gave classical semidegree conditions for long
directed paths and cycles.  Darbinyan and Karapetyan~\cite{DK} studied
short paths, including versions with forbidden vertices.  Their result
bounds directed distance but does not guarantee a path whose length
belongs to $\{3,4,5\}$.  Zhou and Yan~\cite{ZY} proved an $H$-linkage
theorem at the $3n/8$ scale for fixed $H$, sufficiently large order,
and prescribed subdivision-path lengths at least four.  Its semidegree
condition is above the one-third scale considered here.

The same short-linking lemma improves the large-order hypothesis in the
full prescribed-vertex theorem of Kelly, K\"uhn and Osthus, without the
restriction $3\mid\ell$.  We prove below that, for every $\ell\ge7$,
the condition
\[
  n\ge15\ell-60,
  \qquad
  \delta^0(G)\ge\left\lfloor\frac n3\right\rfloor+1
\]
forces a copy of $C_\ell$ through every prescribed vertex.  Their direct
arguments for $\ell=4,5,6$ already require only $n\ge\ell$.  Thus the
same theorem has an explicit linear order bound for every length.
To the best of our knowledge, no previous result improves the explicit
condition $n\ge10^{10}\ell$ while retaining both the prescribed vertex
and the assumption
$\delta^0(G)\ge\lfloor n/3\rfloor+1$.
For $\ell=3q$ and $n=3m$, this corollary still assumes
$\delta^0(G)\ge m+1$.  The equality analysis lowers this to $m$; the
deletion budgets for $q=2$, $q=3$, and $q\ge4$ give the three values
in~\eqref{eq:n0-def}.  If $q\ge3$ and $3\nmid n$,
Corollary~\ref{cor:nonzero-residue} gives the stronger order bound
$n\ge45q-60$.

Without prescribing a vertex, Czygrinow, Molla, Nagle and
Oursler~\cite{CMNO} proved that, for each fixed $\ell\ge4$ and all
sufficiently large $n$, the one-sided condition
\[
  \delta^+(G)\ge\frac{n+1}{3}
\]
forces a copy of $C_\ell$.  Their theorem does not give a rooted
conclusion or an explicit order bound.

The $3n/8$ Hamiltonian threshold was developed
in~\cite{Hag,KKOdirac} and determined exactly for large order
in~\cite{KKOexact}.  At this denser scale, Kelly, K\"uhn and
Osthus~\cite{KKO} proved that every prescribed vertex lies
on a directed cycle of every length $4\le \ell\le n$, while Wang,
Wang and Zhang~\cite{WWZ} treat arbitrary
orientations without a prescribed vertex.  These results do not apply
at semidegree $n/3$.

The restriction $q\ge2$ is essential.  A directed triangle through a
prescribed vertex has a different threshold: Kelly, K\"uhn and
Osthus~\cite{KKO} showed that its asymptotic scale is
$2n/5$, rather than $n/3$.  Thus $C_3$ does not belong to the phenomenon
studied here.

The rooted problem is different from ordinary, unrooted containment.
Kelly, K\"uhn and Osthus~\cite[Conjecture~5]{KKO} conjectured that, if
$\ell\ge4$ and $k>2$ is the smallest integer that does not divide
$\ell$, then, for all sufficiently large $n$, the condition
\[
  \delta^0(G)\ge \left\lfloor\frac nk\right\rfloor+1
\]
forces a copy of $C_\ell$ in every oriented graph $G$ on $n$ vertices.
The conjectured threshold is suggested by cyclic blow-ups.
For $k\ge7$ and $\ell\ge10^7k^6$, K\"uhn, Osthus and
Piguet~\cite{KOP} proved the corresponding asymptotic
result: for every $\eta>0$, the condition
$\delta^0(G)\ge(1+\eta)n/k$ suffices for all sufficiently large $n$.
Grzesik and Volec~\cite[Theorem~1.4]{GV} later determined the corrected
exact thresholds for this unrooted problem.  The conjectured bound is
exact when $3\nmid\ell$ or $\ell\equiv3\pmod{12}$; the other congruence
classes require the rounding corrections stated in their theorem.
In particular, when $3\mid\ell$ and $\ell\ne3$, one has $k\ge4$, so
the unrooted threshold is at most $n/4+O(1)$.
Requiring the cycle to pass through an arbitrary vertex restores the
one-third barrier, as the lower construction
behind~\eqref{eq:old-window} shows.  Related prescribed-vertex results
under the additional assumption that the oriented graph contains no
directed triangle were obtained by Ji, Wu and Song~\cite[Theorem~1.4]{JWS}.

Only the case $n=3m$ requires an additional argument.  Fix the
prescribed vertex $x$.  If $N^+(x)$ is independent,
Lemma~\ref{lem:balanced-cut} gives the root-dominating balanced cut.
Otherwise choose an arc $h\to z$ in $G[N^+(x)]$.  If $N^+(z)$ is
independent, the same lemma gives the transitive-entry balanced cut; if
not, an arc in $G[N^+(z)]$ produces an $xy$-butterfly.  In the two
balanced-cut cases, the row-family classification and matching
arguments are used in Propositions~\ref{prop:first-port}
and~\ref{prop:second-port}.  The butterfly case is completed by
Lemmas~\ref{lem:butterfly-C6} and~\ref{lem:butterfly-long}.  When
$3\nmid n$, every outneighbourhood is non-independent at the threshold
$\delta^0(G)\ge\lceil n/3\rceil$, so only the butterfly case is needed.
Sections~\ref{sec:prelim}--\ref{sec:main-proof} follow this order.

\section{Short linking and preliminary reductions}\label{sec:prelim}

For an oriented graph $G$, write $V(G)$ for its vertex set and
$|G|=|V(G)|$ for its order.  For $v\in V(G)$, let $N_G^+(v)$ and
$N_G^-(v)$ be its out- and inneighbourhoods, and put
$d_G^\pm(v)=|N_G^\pm(v)|$.  We omit the subscript $G$ when the ambient
graph is clear.  For $X\subseteq V(G)$, define
\[
  N^\pm(X)=\bigcup_{x\in X}N^\pm(x).
\]
Thus the external neighbourhood of $X$ is $N^\pm(X)\setminus X$.
For $Y\subseteq V(G)$, set
\[
  N_Y^\pm(v)=N^\pm(v)\cap Y,\qquad
  d_Y^\pm(v)=|N_Y^\pm(v)|.
\]
We write $G[X]$ for the subgraph induced by $X$, $e(X)$ for its number
of arcs, and, for disjoint sets $X,Y$, $e(X,Y)$ for the number of arcs
directed from $X$ to $Y$.  A vertex string $v_0v_1\cdots v_k$ denotes
the directed path
\[
  v_0\to v_1\to\cdots\to v_k.
\]
We use the following two elementary facts repeatedly.

\begin{lemma}\label{lem:elementary}
Let $G$ be an oriented graph.
\begin{enumerate}
\item Every nonempty $X\subseteq V(G)$ satisfies $e(X)\le\binom{|X|}{2}$.  Consequently,
\[
  |N^+(X)\setminus X|\ge \delta^+(G)-\frac{|X|-1}{2},
\]
and the analogous reverse inequality holds for $N^-(X)\setminus X$.
\item Every independent set has size at most $|G|-2\delta^0(G)$.
\end{enumerate}
\end{lemma}

\begin{proof}
The first assertion follows because an oriented graph has at most one arc on each unordered pair.  Choose a vertex of $G[X]$ with outdegree at most $(|X|-1)/2$ to obtain the displayed inequality; reverse all arcs for its counterpart.  If $X$ is independent, then for any $x\in X$ the sets $N^+(x)$ and $N^-(x)$ are disjoint subsets of $V(G)\setminus X$, giving $|G|-|X|\ge2\delta^0(G)$.
\end{proof}

Kelly, K\"uhn and Osthus~\cite{KKO} proved that, for a
positive integer $C$, under the hypotheses
\[
  \delta^0(H)\ge N/3-C+1
  \qquad\text{and}\qquad
  N\ge8\cdot10^9C,
\]
every ordered pair of distinct vertices is joined by a path whose
length belongs to $\{3,4,5\}$.  The next lemma replaces their
large-order hypothesis by a sharp linear inequality.

\begin{lemma}\label{lem:short-link}
Let $H$ be an oriented graph of order $N$, and put
$d=\delta^0(H)$.  If $d\ge3$ and
\[
  7d\ge2N+3,
\]
then every ordered pair of distinct vertices $x,y$ is joined by an
$x$--$y$ path of length $3$, $4$, or $5$.
\end{lemma}

\begin{proof}
The proof idea has three steps.  We first choose equal-sized sets in the
outneighbourhood of the initial vertex and the inneighbourhood of the
terminal vertex; the absence of paths of lengths $3$, $4$, and $5$
then forces a rigid system of forbidden arcs between the resulting
layers.  Minimum semidegree makes two intermediate layers large, while
the forbidden arcs give an incompatible upper bound on the total
outdegree of one of them.  The contradiction reduces to a quadratic
inequality.  Its two endpoint estimates use precisely the relation
$7d\ge2N+3$.

\smallskip
\noindent\emph{Step 1: the forbidden-arc structure.}
Put $s=d-2$.
The sets $N^+(x)\setminus\{y\}$ and $N^-(y)\setminus\{x\}$ both have
size at least $s+1$.  Choose an $s$-set
$X\subseteq N^+(x)\setminus\{y\}$.  Since
$|N^-(y)\setminus\{x\}|\ge s+1>|X|$, there is an $s$-set
$Y\subseteq N^-(y)\setminus\{x\}$ with $Y\ne X$.  Set
\[
  Z=X\cap Y,\qquad A=X\setminus Y,\qquad B=Y\setminus X,
  \qquad a=|A|=|B|\ge1.
\]
Suppose, for a contradiction, that there is no $x$--$y$ path of
length $3$, $4$, or $5$.  Then there is no arc from $X$ to $Y$.
Define
\[
  P=N^+(A)\setminus(X\cup\{x,y\}),\qquad
  Q=N^-(B)\setminus(Y\cup\{x,y\}),
\]
and put $p=|P|$ and $u=|Q|$.  The sets
$X\cup Y$, $P$, $Q$, and $\{x,y\}$ are pairwise disjoint.  Indeed,
$P\cap Y\ne\varnothing$ or $Q\cap X\ne\varnothing$ would give an
$x$--$y$ path of length three, while $P\cap Q\ne\varnothing$ would
give one of length four.  We also have
\[
  P\nrightarrow Y,\qquad P\nrightarrow Q,
  \qquad P\nrightarrow y;
\]
otherwise there is an $x$--$y$ path of length four, five, or three,
respectively.

\smallskip
\noindent\emph{Step 2: degree counting.}
Let
\[
  R=V(H)\setminus\bigl((X\cup Y)\cup P\cup Q\cup\{x,y\}\bigr),
  \qquad \rho=|R|=N-d-a-p-u,
\]
and define
\[
  L=d-\frac{a+1}{2}.
\]
Since $a\le s=d-2$, we have $L\ge(d+1)/2>0$.  No arc goes from $A$ to
$Y$, and $A\nrightarrow x$ because $x\to A$.  Hence every out-arc
of $A$ goes inside $A$, to $P$, or possibly to $y$, and therefore
\[
  e(A,P)\ge ad-\binom a2-a=aL.
\]
The reverse argument, applied to the indegrees of $B$, gives
$e(Q,B)\ge aL$.  Consequently,
\[
  p,u\ge L.
\]

All out-arcs of $P$ go to $A$, inside $P$, to $R$, or possibly to
$x$.  Orientedness between $A$ and $P$ gives
$e(P,A)\le ap-e(A,P)\le ap-aL$.  Summing the outdegrees of the vertices of
$P$, we obtain
\begin{align*}
  pd
  &\le e(P,A)+e(P)+e(P,R)+e(P,\{x\})\\
  &\le ap-aL+\binom p2+p\rho+p.
\end{align*}
Thus
\[
  0\ge F:=p\left(d-a-\rho-\frac{p+1}{2}\right)+aL.
\]
Substituting the value of $\rho$ and using $u\ge L$ gives
\[
  F\ge \frac{p^2}{2}
       -p\left(\frac a2+c\right)+aL
   =:h(p,a),
  \qquad c:=N-3d+1.
\]
The hypothesis $7d\ge2N+3$ is equivalent to
\begin{equation}\label{eq:d-c-relation}
  d\ge2c+1.
\end{equation}

\smallskip
\noindent\emph{Step 3: the quadratic contradiction.}
Write $\kappa=a/2+c$.  If $L\ge\kappa$, then $h$ is increasing in
$p$ for $p\ge L$, and
\[
  h(p,a)\ge h(L,a)
  =L\left(\frac{L+a}{2}-c\right)>0.
\]
Indeed, $L+a=d+(a-1)/2\ge d$, so the last factor is at least
$d/2-c\ge1/2$ by~\eqref{eq:d-c-relation}.

It remains that $L<\kappa$.  For fixed $a$, the convex quadratic
$h(p,a)$ is minimized at $p=\kappa$, so
\[
  h(p,a)\ge
  g(a):=a\left(d-\frac{a+1}{2}\right)
  -\frac12\left(\frac a2+c\right)^2.
\]
Since $a$ is an integer, $L<\kappa$ is equivalent to
$a\ge d-c$.  Together with $a\le d-2$, this implies $c\ge2$.
The function $g$ is concave in $a$, so it suffices to check the two
endpoints of
\[
  d-c\le a\le d-2.
\]
Write $t=d-2c-1\ge0$.  Direct expansion gives
\begin{align*}
  8g(d-c)
  &=3c^2+6c-1+(10c+2)t+3t^2>0,\\
  8g(d-2)
  &=16c-9+(8c+6)t+3t^2>0.
\end{align*}
Thus $g(a)>0$ throughout the interval, again contradicting $F\le0$.
\end{proof}

Substituting $d\ge N/3-C+1$ in Lemma~\ref{lem:short-link} gives the
following form.

\begin{corollary}\label{cor:defect-link}
Let $C$ be a positive integer.  If $H$ is an oriented graph on
$N\ge21C-12$ vertices with
\[
  \delta^0(H)\ge\frac N3-C+1,
\]
then every ordered pair of distinct vertices is joined by a path of
length $3$, $4$, or $5$.
\end{corollary}

\begin{proof}
Put $d=\delta^0(H)$.  The hypotheses give
\[
  d\ge\frac N3-C+1\ge6C-3\ge3
\]
and
\[
  7d\ge7\left(\frac N3-C+1\right)\ge2N+3.
\]
Apply Lemma~\ref{lem:short-link}.
\end{proof}

The following construction shows that $+3$ cannot be reduced and that
the coefficient $21$ is asymptotically best possible.

\begin{proposition}\label{prop:link-sharp}
For every integer $c\ge2$, there is an oriented graph $H_c$ of order
$N=7c-1$ with $\delta^0(H_c)=2c$ and vertices $x,y$ such that no
$x$--$y$ path has length $3$, $4$, or $5$.  Thus the constant $+3$
in Lemma~\ref{lem:short-link} cannot be replaced by $+2$.  Moreover,
the coefficient $21$ in Corollary~\ref{cor:defect-link} is
asymptotically best possible.
\end{proposition}

\begin{proof}
Take disjoint sets $S,P,Q,R$ with
\[
  |S|=|P|=|Q|=2c-1,\qquad |R|=c,
\]
together with two vertices $x,y$.  Let $P$ and $Q$ induce regular
tournaments, and let $S$ and $R$ be independent.  Add all arcs
indicated by
\[
  x\to y,\qquad x\to S\to y,
\]
\[
  S\to P\to R\to Q\to S,
\]
and
\[
  y\to P\cup Q\cup R,
  \qquad P\cup Q\cup R\to x,
\]
with no other arcs between the parts.  Regular tournaments exist on
$2c-1$ vertices.

The vertices in $\{x\}$, $\{y\}$, $S$, $P$, $Q$, and $R$ have
respective $(d^-,d^+)$-pairs
\[
  (5c-2,2c),\quad (2c,5c-2),\quad (2c,2c),
  \quad (3c-1,2c),\quad (2c,3c-1),\quad (2c,2c).
\]
Hence $|H_c|=7c-1$ and $\delta^0(H_c)=2c$, so
\[
  7\delta^0(H_c)=14c=2|H_c|+2.
\]

Every $x$--$y$ path is either the arc $xy$, has the form $xsy$ with
$s\in S$, or begins $x\to S\to P$.  In the last case, any simple
path ending at $y$ must pass successively through
\[
  P,\ R,\ Q,\ S,\ y;
\]
arcs inside the regular tournaments can only increase its length.
Thus every remaining $x$--$y$ path has length at least six.

Set
\[
  C_c=\left\lceil\frac{c+2}{3}\right\rceil.
\]
Then $2c\ge(7c-1)/3-C_c+1$, while
$(7c-1)/C_c\to21$.  Consequently, there are no constants
$\gamma<21$ and $K$ for which the condition $N\ge\gamma C+K$ is a
universal sufficient order hypothesis in Corollary~\ref{cor:defect-link}.
\end{proof}

We next record the extension and deletion statements used below.

\begin{lemma}\label{lem:extension-deletion}
Let $G$ be an oriented graph.
\begin{enumerate}
\item Let $d$ and $L$ be nonnegative integers.  If
$\delta^+(G)\ge d$, $F\subseteq V(G)$, $v\notin F$, and
\[
  |F|+L\le d,
\]
then $G$ contains a simple path of length $L$ starting at $v$ and otherwise avoiding $F$.
\item Let $G$ have order $n$ and minimum semidegree at least
$\lceil n/3\rceil$, and let $R\ge0$ be an integer.  If
\[
  n\ge15R+7
\]
and $H$ is obtained by deleting at most $R$ vertices, then every
ordered pair of distinct vertices of $H$ is joined in $H$ by a path
of length $3$, $4$, or $5$.
\item Let $G$ have order $n$ and minimum semidegree at least
$\lceil n/3\rceil$.  Let $\ell$ and $B$ be integers with
$\ell\ge B\ge6$, let $x,a$ be distinct
vertices, and let $Z\subseteq V(G)\setminus\{x,a\}$.  Suppose that,
for each $j\in\{B-5,B-4,B-3\}$, there is an $x$--$a$ path $Q_j$ of
length $j$ whose internal vertices lie in $Z$.  Put
\[
  R=|Z|+\ell-B.
\]
If
\[
  |Z|+1+\ell-B\le\left\lceil\frac n3\right\rceil
  \qquad\text{and}\qquad
  n\ge15R+7,
\]
then $x$ lies on a copy of $C_\ell$.
\end{enumerate}
\end{lemma}

\begin{proof}
For (i), extend greedily.  Before the $(j+1)$st step, at most $|F|+j$ vertices are forbidden, which is smaller than $d$ for $0\le j<L$.

For (ii), let $r\le R$ vertices be deleted, put $N=|H|=n-r$,
and write $n=3m+s$ with $s\in\{0,1,2\}$.  Set
\[
  D=\left\lceil\frac n3\right\rceil
  =\begin{cases}
     m,&s=0,\\
     m+1,&s=1,2.
   \end{cases}
\]
Then $\delta^0(H)\ge D-r$.  A direct calculation shows that
\begin{equation}\label{eq:deletion-residue}
7(D-r)-\bigl(2(n-r)+3\bigr)
=\begin{cases}
  m-5r-3,&s=0,\\
  m-5r+2,&s=1,\\
  m-5r,&s=2.
\end{cases}
\end{equation}
The hypothesis $n\ge15R+7$ gives
\[
  m\ge
  \begin{cases}
    5R+3,&s=0,\\
    5R+2,&s=1,2.
  \end{cases}
\]
Since $r\le R$, the three expressions in
\eqref{eq:deletion-residue} are respectively at least $0$, $4$, and
$2$.  Moreover, $D-r\ge D-R\ge4R+3\ge3$.
Lemma~\ref{lem:short-link} now applies to $H$.

For (iii), apply (i) with forbidden set $Z\cup\{x\}$ to obtain a path
$P$ of length $\ell-B$ from $a$ to a vertex $y$.  Delete $Z$ and
$V(P)\setminus\{y\}$.  At most $R$ vertices are deleted, so (ii)
gives a $y$--$x$ path of some length $t\in\{3,4,5\}$.  The paths
$Q_{B-t}$, $P$, and the $y$--$x$ path are internally disjoint and
together form a $C_\ell$.
\end{proof}

\begin{remark}\label{rem:constant-track}
The three cases in~\eqref{eq:n0-def} come from three different deletion
budgets.  For $q=2$, the only critical use of short linking deletes
$R=\ell-4=2$ vertices in the transitive-entry branch, giving
$15R+7=37$.  For $q=3$, a terminal-safe three-chain is already a
$C_9$; the largest remaining budget is $R=\ell-3=6$ in the
distinct-row subcase of the transitive-entry branch, giving
$15R+7=97$.  For $q\ge4$, the terminal-safe-chain branch of
Proposition~\ref{prop:first-port} may delete
$R=\ell-1=3q-1$ vertices, giving $15R+7=45q-8$.
These are the smallest uniform cutoffs delivered by the present
argument.  The stable balanced-cut branches use only the matching
inequality $m\ge2q+11$.
\end{remark}

Following Kelly, K\"uhn and Osthus
~\cite[discussion preceding Fact~18]{KKO}, an $xy$-\emph{butterfly}
consists of five distinct vertices $x,h,z,b,y$ and the six arcs
\[
  xh,\ xz,\ hz,\ zb,\ zy,\ by.
\]
It contains $x$--$y$ paths of lengths $2$, $3$, and $4$.

The sharp linking lemma also gives a linear order bound in the full
prescribed-vertex theorem of Kelly, K\"uhn and Osthus.

\begin{corollary}
\label{cor:kko-linear}
Let $\ell\ge7$ and $n\ge15\ell-60$.  If $G$ is an oriented graph on
$n$ vertices with
\[
  \delta^0(G)\ge\left\lfloor\frac n3\right\rfloor+1,
\]
then every vertex of $G$ lies on a copy of $C_\ell$.
\end{corollary}

\begin{proof}
Fix $x\in V(G)$ and put $D=\lfloor n/3\rfloor+1$.  By
Lemma~\ref{lem:elementary}(ii), every independent set has size at most
$n-2D<D$.  Hence $N^+(x)$ contains an arc $h\to z$, and $N^+(z)$
contains an arc $b\to y$.  The five vertices are distinct, and
\[
  xh,xz,hz,zb,zy,by
\]
form an $xy$-butterfly.
Indeed, $b,y\notin\{x,h\}$ because $x\to z$ and $h\to z$, whereas
$z\to b$ and $z\to y$.
This is the argument of~\cite[Fact~18]{KKO}; we include it to keep
track of the order bound.

Since $n\ge15\ell-60$ and $\ell\ge7$, we have
\[
  4+(\ell-7)=\ell-3\le D.
\]
Lemma~\ref{lem:extension-deletion}(i) gives a path $P$ of length
$\ell-7$ from $y$ to a vertex $v$, otherwise avoiding
$\{x,h,z,b\}$.  Delete
\[
  \{h,z,b\}\cup(V(P)\setminus\{v\})
\]
and call the remaining graph $H$.  Exactly $r=\ell-4$ vertices are
deleted.

Write $n=3m+s$ with $s\in\{0,1,2\}$.  Then
\[
  |H|=3m+s-r,
  \qquad
  \delta^0(H)\ge m+1-r.
\]
Moreover, $n\ge15\ell-60$ gives $m\ge5\ell-20=5r$ since $s\le2$, and hence
\[
  7(m+1-r)-\bigl(2(3m+s-r)+3\bigr)
  =m-5r+4-2s\ge0.
\]
Also $m+1-r\ge4\ell-15\ge3$.  Lemma~\ref{lem:short-link} therefore
gives a $v$--$x$ path in $H$ of some length
$t\in\{3,4,5\}$.  Choose the $x$--$y$ path of length $7-t$ in the
butterfly.  Together with $P$ and the $v$--$x$ path, it forms a simple
cycle of length
\[
  (7-t)+(\ell-7)+t=\ell.
\]
\end{proof}

Together with Lemmas~16, 17 and 19 of Kelly, K\"uhn and
Osthus~\cite{KKO}, Corollary~\ref{cor:kko-linear} replaces the order
hypothesis in their Theorem~4 by $n\ge\ell$ for $\ell=4,5,6$, and by
$n\ge15\ell-60$ for $\ell\ge7$.

Taking $\ell=3q$ in Corollary~\ref{cor:kko-linear} gives the following
bound when $3\nmid n$.

\begin{corollary}\label{cor:nonzero-residue}
Let $q\ge3$, let $3\nmid n$, and suppose $n\ge45q-60$.  If $G$ is an
oriented graph on $n$ vertices with
\[
  \delta^0(G)\ge\left\lceil\frac n3\right\rceil,
\]
then every vertex of $G$ lies on a copy of $C_{3q}$.
\end{corollary}

\begin{proof}
Since $3\nmid n$,
\[
  \left\lceil\frac n3\right\rceil
  =\left\lfloor\frac n3\right\rfloor+1.
\]
Apply Corollary~\ref{cor:kko-linear} with $\ell=3q$.
\end{proof}

We next isolate the equality structure produced by an independent outneighbourhood.

\begin{lemma}\label{lem:balanced-cut}
Let $G$ be an oriented graph on $3m$ vertices with $\delta^0(G)\ge m$.  If $N^+(z)$ is independent, then, with
\[
  A=N^+(z),\qquad B=V(G)\setminus A,
\]
we have $|A|=m$, $|B|=2m$, and
\[
  d_B^+(a)=d_B^-(a)=m
  \qquad\text{for every }a\in A.
\]
In particular, each pair in $A\times B$ is joined by exactly one arc.
\end{lemma}

\begin{proof}
Lemma~\ref{lem:elementary}(ii) gives $|A|\le m$, while $d^+(z)\ge m$ gives $|A|\ge m$.  Thus $|A|=m$.  Since $A$ is independent, every in- and outneighbour of $a\in A$ lies in $B$.  The two neighbourhoods are disjoint, have size at least $m$ each, and lie in the $2m$-set $B$, so equality holds throughout.
\end{proof}

We call a partition $(A,B)$ of $V(G)$ a \emph{balanced cut} if, for
some integer $m$, the set $A$ is independent, $|A|=m$, $|B|=2m$, and
\[
  d_B^+(a)=d_B^-(a)=m
  \qquad\text{for every }a\in A.
\]
If $A=N^+(z)$, we say that the balanced cut is \emph{rooted at $z$}.

In a balanced cut, put
\[
  O(a):=N_B^+(a)\qquad(a\in A).
\]
Thus every row $O(a)$ has size $m$.  For $b\in B$, define
\[
  s_b:=|N_A^-(b)|=|\{a\in A:a\to b\}|,
  \qquad t_b:=m-s_b=|N_A^+(b)|.
\]
The semidegree condition gives the fundamental inequalities
\begin{equation}\label{eq:bookkeeping}
  d_B^+(b)\ge s_b,\qquad d_B^-(b)\ge t_b,\qquad
  \sum_{b\in B}s_b=\sum_{b\in B}t_b=m^2.
\end{equation}

\section{Nonextendable endpoints and row-family stability}\label{sec:stability}

Throughout the first part of this section, $(A,B)$ is a balanced cut and $x\in B$ satisfies $x\to A$.  Thus $x\notin O(a)$ for every $a\in A$.  A \emph{$D_3$-transition} into $a$ is a simple path
\[
  a'\to b\to c\to a,
  \qquad a'\in A\setminus\{a\},\quad b,c\in B\setminus\{x\}.
\]

\begin{lemma}\label{lem:dead}
If $m\ge3$, then at most three vertices of $A$ admit no $D_3$-transition.
\end{lemma}

\begin{proof}
Call such a vertex $a$ \emph{nonextendable} and put
\[
  J_a=N_B^-(a)\setminus\{x\},\qquad |J_a|=m-1,
\]
and
\[
  L_a=\{b\in B\setminus\{x\}:N_A^-(b)\subseteq\{a\}\}.
\]
If $c\in J_a$, then every inneighbour of $c$ in $B\setminus\{x\}$ lies in $L_a$; otherwise some $b\to c$ has an inneighbour $a'\in A\setminus\{a\}$, giving $a'\to b\to c\to a$.

\smallskip
\noindent\emph{Nonextendable vertices with $L_a=\varnothing$.}
For $c\in J_a$, the only possible inneighbour of $c$ in $B$ is $x$,
so $d_B^-(c)\le1$.  By~\eqref{eq:bookkeeping}, $t_c\le1$, while
$c\to a$ gives $t_c\ge1$.  Hence $t_c=1$ and $a$ is the unique
outneighbour of $c$ in $A$.  It follows that the sets $J_a$ belonging
to distinct nonextendable vertices with $L_a=\varnothing$ are pairwise
disjoint.  Since three such sets have total size
$3(m-1)>2m-1=|B\setminus\{x\}|$, there are at most two nonextendable
vertices of this kind.

\smallskip
\noindent\emph{Nonextendable vertices with $L_a\ne\varnothing$.}
We show that there is at most one.  Suppose that distinct nonextendable
vertices $a,a'$ have nonempty $L_a,L_{a'}$.  First, these sets are
disjoint.  Indeed, if $L_0=L_a\cap L_{a'}$ is nonempty, then every
$z\in L_0$ has $s_z=0$ and $t_z=m$.  Moreover,
$z\in J_a\cap J_{a'}$, and all its inneighbours in
$B\setminus\{x\}$ lie in $L_0$.  Therefore every $z\in L_0$ has at
least $m-1$ inneighbours in $L_0$, and
\[
  |L_0|(m-1)\le e(L_0)\le\binom{|L_0|}{2}.
\]
Thus $|L_0|\ge2m-1$.  On the other hand, $t_x=m$ and $\sum_{b\in B}t_b=m^2$, so $(|L_0|+1)m\le m^2$, giving $|L_0|\le m-1$, a contradiction.

Write $k=|L_a|$ and $k'=|L_{a'}|$.  For $z\in L_a$, we have $s_z\le1$ and hence $t_z\ge m-1$.  Since $a'\ne a$, the definition of $L_a$ gives $z\to a'$, so $z\in J_{a'}$.  Nonextendability of $a'$ implies that at least $m-2$ inneighbours of $z$ lie in $L_{a'}$.  Consequently,
\[
  e(L_{a'},L_a)\ge k(m-2).
\]
Symmetrically, $e(L_a,L_{a'})\ge k'(m-2)$.  The two sets are disjoint and the graph is oriented, so at most one of the two possible arcs occurs on each pair in $L_a\times L_{a'}$.  Therefore
\[
  (k+k')(m-2)
  \le e(L_{a'},L_a)+e(L_a,L_{a'})
  \le kk'.
\]
Furthermore,
\[
  m+(k+k')(m-1)\le\sum_{b\in B}t_b=m^2,
\]
so $k+k'\le m$.  Consequently,
\[
  m-2\le\frac{kk'}{k+k'}\le\frac{k+k'}4\le\frac m4,
\]
which is impossible for $m\ge3$.  Hence at most one nonextendable vertex has nonempty $L_a$, and the total number of nonextendable vertices is at most three.
\end{proof}

We now turn to the set-system statement.  Let $U$ be a finite set, let $C\subseteq U$, and let $\calO=(O_i:i\in I)$ be a family of equal-sized subsets of $U$, each meeting $C$.  A \emph{terminal-safe three-chain} consists of four distinct indices $i_0,i_1,i_2,i_3$ and four pairwise distinct elements $b_0,b_1,b_2,d$ such that
\[
  b_j\in O_{i_j}\setminus O_{i_{j+1}}\quad(0\le j\le2),
  \qquad d\in O_{i_3}\cap C.
\]

\begin{lemma}\label{lem:rainbow}
Let $S_1,S_2,S_3,S_4$ be four distinct $k$-subsets of a set $U$.
There is a cyclic ordering of these sets and pairwise distinct elements
\[
  z_i\in S_i\setminus S_{i+1}\qquad(i\in\mathbb Z/4\mathbb Z).
\]
\end{lemma}

\begin{proof}
For a cyclic order $T_1,T_2,T_3,T_4$, put
\[
  D_i=T_i\setminus T_{i+1}
  \qquad(i\in\mathbb Z/4\mathbb Z).
\]
We have $D_i\cap D_{i+1}=\varnothing$: membership in $T_{i+1}$
is required by $D_{i+1}$ and forbidden by $D_i$.  Consequently,
\[
  (D_1\cup D_3)\cap(D_2\cup D_4)=\varnothing.
\]
Thus the four sets $D_i$ have a system of distinct representatives if
and only if each opposite pair $(D_1,D_3)$ and $(D_2,D_4)$ has two
distinct representatives; representatives chosen for different opposite
pairs are automatically distinct.  Each $D_i$ is nonempty, because
consecutive sets have the same size and are distinct.  Hence an opposite
pair fails Hall's condition precisely when its two members are the same
singleton.

Write
\[
  \Delta(P,Q)=|P\setminus Q|=|Q\setminus P|
\]
and call a pair $ij$ \emph{coarse} if $\Delta(S_i,S_j)\ge2$.  If two coarse pairs are adjacent, extend them to a Hamilton cycle on the four labels.  The two coarse directed differences belong to different opposite pairs, so neither opposite pair can be the same-singleton obstruction.

It remains to consider the case where the graph of coarse pairs is a matching.  Suppose first that it is nonempty, and relabel so that $12$ is coarse.  Then $13,14,23,24$ are thin, meaning that their $\Delta$-value is one.  Consider the cycles
\[
  1234,\quad1432,\quad1243,\quad1342.
\]
In each order, the coarse edge $12$ supplies a difference set of size at
least two to one opposite pair, so only the other opposite pair can
obstruct a rainbow choice.  If all four cycles fail, the four obstructions
are as follows:
\[
\begin{array}{c|c}
\text{cyclic order}&\text{same-singleton obstruction}\\ \hline
1234&S_2\setminus S_3=S_4\setminus S_1=\{a\}\\
1432&S_1\setminus S_4=S_3\setminus S_2=\{b\}\\
1243&S_2\setminus S_4=S_3\setminus S_1=\{c\}\\
1342&S_1\setminus S_3=S_4\setminus S_2=\{d\}
\end{array}
\]
Here $a,c\in S_2$ and $b,d\notin S_2$.  Also $a\ne c$ and $b\ne d$, since either equality would impose contradictory membership in $S_4$ or $S_3$, respectively.  Thus $a,b,c,d$ are pairwise distinct, and
\[
  S_3=(S_2\setminus\{a\})\cup\{b\},\qquad
  S_4=(S_2\setminus\{c\})\cup\{d\},
\]
and
\[
  S_1=(S_2\setminus\{a,c\})\cup\{b,d\}.
\]
For the cycle $1324$, the four directed differences contain, in order, $d,b,c,a$, a contradiction.

Finally, suppose there is no coarse pair.  Write
\[
  S_1=K\cup\{p\},\qquad S_2=K\cup\{q\}.
\]
Every $k$-set $S\notin\{S_1,S_2\}$ satisfying
$\Delta(S,S_1)=\Delta(S,S_2)=1$ has exactly one of the forms
\[
  K\cup\{r\},\quad r\notin K\cup\{p,q\},
  \qquad\text{or}\qquad
  (K\setminus\{r\})\cup\{p,q\},\quad r\in K.
\]
A set of the first form and one of the second form have $\Delta$-value
two.  Hence $S_3,S_4$ have the same form, and all four sets are either
$K\cup\{r_i\}$ or $L\setminus\{r_i\}$, where
$L:=K\cup\{p,q\}$.  In the first case any cyclic order has the
distinct witnesses $r_i$; in the second it has the distinct witnesses
$r_{i+1}$.
\end{proof}

\begin{corollary}\label{cor:four-types}
If four distinct row types occur in $\calO$, then $\calO$ has a terminal-safe three-chain.
\end{corollary}

\begin{proof}
Take a rainbow cycle $T_0T_1T_2T_3T_0$ with witnesses $z_i\in T_i\setminus T_{i+1}$.  If some $z_j\in C$, delete the outgoing edge $T_jT_{j+1}$, order the remaining path as
\[
  T_{j+1},T_{j+2},T_{j+3},T_j,
\]
and use $z_j$ as the terminal representative.  If no $z_i$ lies in $C$, delete any edge and choose an arbitrary element of the final row in $C$ as the terminal representative.  In either case the four representatives are distinct.
\end{proof}

The preceding lemma yields the following classification.

\begin{proposition}\label{prop:classification}
Let $|I|\ge4$, and let every row $O_i$ be a $k$-subset of $U$ meeting $C$.  The family has no terminal-safe three-chain if and only if one of the following holds.
\begin{enumerate}
\item There is one row type.
\item There are exactly two row types $X,Y$, and either one has multiplicity one, or both have multiplicity at least two and $|X\setminus Y|=1$.
\item There are exactly three row types of one of the following forms:
\[
  K\cup\{p\},\quad K\cup\{q\},\quad K\cup\{r\},
\]
or
\[
  L\setminus\{p\},\quad L\setminus\{q\},\quad L\setminus\{r\},
\]
where $p,q,r$ are distinct, and the common intersection of the three rows is disjoint from $C$.
\end{enumerate}
\end{proposition}

\begin{proof}
Corollary~\ref{cor:four-types} excludes four row types.  If there is one
row type, then every difference
$O_{i_j}\setminus O_{i_{j+1}}$ is empty, so no terminal-safe
three-chain exists.

Suppose there are two types $X,Y$.  A four-index chain exists only if both types occur at least twice, and then the type sequence must alternate.  Put $P=X\setminus Y$ and $Q=Y\setminus X$.  If $|P|\ge2$, choose two distinct representatives from the two copies of $P$, choose $d\in Y\cap C$, and choose a representative from $Q$ distinct from $d$; this is possible because $|Q|=|P|\ge2$.  Conversely, if $|P|=1$, the two $P$-positions cannot receive distinct representatives.  This proves (ii).

Now suppose there are three types.  Some type, say $X$, occurs at least twice; call the other types $Y,Z$.  For the type sequence $Y,X,Z,X$, the four representative sets are
\[
  Y\setminus X,\quad X\setminus Z,\quad Z\setminus X,\quad X\cap C.
\]
All four sets are nonempty: the first three because the row types are
distinct and equicardinal, and the last because every row meets $C$.
The first and third lie outside $X$, whereas the second and fourth lie
inside $X$.  Thus the union of the outside pair is disjoint from the
union of the inside pair.  As in Lemma~\ref{lem:rainbow}, an SDR exists
if and only if each pair has two distinct representatives, and a pair
fails precisely when its two members are the same singleton.  Hence
Hall's condition fails exactly when
\begin{equation}\label{eq:outside}
  Y\setminus X=Z\setminus X=\{p\},
\end{equation}
or
\begin{equation}\label{eq:inside-z}
  X\setminus Z=X\cap C=\{q\}.
\end{equation}
For the sequence $Z,X,Y,X$, failure is equivalent to~\eqref{eq:outside} or
\begin{equation}\label{eq:inside-y}
  X\setminus Y=X\cap C=\{q\}.
\end{equation}
Thus either~\eqref{eq:outside} holds, or both~\eqref{eq:inside-z} and~\eqref{eq:inside-y} hold.

In the first case, equal row sizes give
\[
  X=L\setminus\{p\},\qquad Y=L\setminus\{y\},\qquad Z=L\setminus\{z\}
\]
for distinct $p,y,z$.  In the second case,
\[
  X=K\cup\{q\},\qquad Y=K\cup\{p\},\qquad Z=K\cup\{r\}
\]
for distinct $p,q,r$.  Thus the types have one of the two forms in
part~(iii).

In the first form, if the common intersection contained $d\in C$, then
the sequence $X,Y,X,Z$ would have the three distinct hub
representatives together with $d$, giving a terminal-safe three-chain.
In the second form, $X\cap C=\{q\}$ already implies
$K\cap C=\varnothing$.  Conversely, for either family with $C$-free
common intersection, every difference representative and every
terminal representative lies in the same three-point hub.  Four
distinct representatives are impossible.
\end{proof}

\begin{corollary}\label{cor:stability}
If a row family indexed by $I$, $|I|\ge4$, has no terminal-safe three-chain, then either
\begin{enumerate}
\item one row type has multiplicity at least $|I|-1$, or
\item
\[
  \left|\left(\bigcup_{i\in I}O_i\right)
  \setminus\left(\bigcap_{i\in I}O_i\right)\right|\le3.
\]
\end{enumerate}
\end{corollary}

\begin{proof}
In Proposition~\ref{prop:classification}(ii), a singleton type leaves the other type with multiplicity $|I|-1$; otherwise the two types have total variation two.  In case (iii), the total variation is exactly the three-point hub.
\end{proof}

\section{The root-dominating balanced cut}\label{sec:first-port}

In this section, $G$ has order $3m$ and minimum semidegree at least $m$, $(A,B)$ is a balanced cut, and $x\in B$ satisfies $x\to A$.  Put
\[
  C=N_B^-(x),\qquad U=B\setminus\{x\}.
\]
Every row $O(a)$ is an $m$-subset of the $(2m-1)$-set $U$, while $|C|\ge m$.  Hence
\begin{equation}\label{eq:row-meets-C}
  O(a)\cap C\ne\varnothing
  \qquad\text{for every }a\in A.
\end{equation}

We first handle $C_6$ directly.

\begin{proposition}\label{prop:first-C6}
If $m\ge9$, then $x$ lies on a copy of $C_6$.
\end{proposition}

\begin{proof}
Let $W$ be the set of nonextendable endpoints from Lemma~\ref{lem:dead}, so $|W|\le3$.  Suppose first that some $a\in A\setminus W$ satisfies $|O(a)\cap C|\ge2$.  Choose a $D_3$-transition
\[
  a_0\to b\to c\to a.
\]
Since $c\to a$, we have $c\notin O(a)$.  Choose
\[
  d\in (O(a)\cap C)\setminus\{b\}.
\]
Then
\[
  x\to a_0\to b\to c\to a\to d\to x
\]
is a $C_6$.

We may therefore assume that $|O(a)\cap C|=1$ for every $a\in A\setminus W$.  Since there is at least one such row, the inequality
\[
  |O(a)\cap C|\ge |O(a)|+|C|-|U|=|C|-m+1
\]
implies $|C|\le m$.  On the other hand, $x\to A$, so $x$ has no
inneighbour in $A$ and
\[
  |C|=d_B^-(x)=d^-(x)\ge m.
\]
Thus $|C|=m$.  Put $D=U\setminus C$, so $|D|=m-1$.  Every extendable row has the form
\[
  O(a)=D\cup\{d_a\},\qquad d_a\in C.
\]
Thus $s_b\ge m-3$ for every $b\in D$, and~\eqref{eq:bookkeeping} gives
\begin{align}
e(D,C)
&\ge (m-1)(m-3)-\binom{m-1}{2}-(m-1)\notag\\
&=\frac{(m-1)(m-6)}2.\label{eq:C6-DC}
\end{align}
For $m\ge9$,
\[
  \frac{(m-1)(m-6)}2-(m-1)
  =\frac{(m-1)(m-8)}2>0.
\]
We claim that there are $b\in D$, $c\in C$, and $a\in A\setminus W$ such that
$b\to c$ and $d_a\ne c$.  Otherwise, for every arc $b\to c$ from
$D$ to $C$ and every extendable row label $d_a$, we would have
$d_a=c$.  The lower bound~\eqref{eq:C6-DC} is positive, and
$A\setminus W$ is nonempty, so all targets of arcs from $D$ to $C$
and all extendable row labels would equal a single vertex $c_*$.  This
would give
\[
  e(D,C)=e(D,\{c_*\})\le |D|=m-1,
\]
a contradiction.  Since $|A\setminus W|\ge m-3\ge2$, choose another
extendable vertex $a_0\ne a$.  Then
\[
  x\to a_0\to b\to c\to a\to d_a\to x
\]
is a $C_6$.
\end{proof}

We now close all longer multiples of three under the linear order
hypothesis of Theorem~\ref{thm:main}.

\begin{proposition}\label{prop:first-port}
Let $q\ge3$, put $\ell=3q$, and suppose $n=3m\ge n_0(q)$.
Then $x$ lies on a copy of $C_\ell$.
\end{proposition}

\begin{proof}
Since $n=3m\ge n_0(q)$,
\[
  m\ge
  \begin{cases}
    33,&q=3,\\
    15q-2,&q\ge4.
  \end{cases}
\]
In particular,
\begin{equation}\label{eq:first-port-numerics}
  m\ge\ell,\qquad m\ge2q+11,\qquad m-4\ge q.
\end{equation}
Let $W$ be the nonextendable-endpoint set and put $A_0=A\setminus W$.
Thus $|W|\le3$ and every $a\in A_0$ admits a $D_3$-transition.  Apply
Corollary~\ref{cor:stability} to the rows indexed by $A_0$, with terminal
set $C$.

The proof follows the three outcomes in Corollary~\ref{cor:stability}.
A terminal-safe three-chain gives three possible initial lengths and is
closed by short linking.  In each of the two remaining outcomes,
\eqref{eq:bookkeeping} gives a dense bipartite graph, and K\"onig's
theorem supplies the required matching.

\smallskip
\noindent\emph{Case 1: a terminal-safe three-chain.}
Suppose first that there is a terminal-safe three-chain.  It yields distinct vertices with
\begin{equation}\label{eq:safe-chain}
  x\to a_0\to b_0\to a_1\to b_1\to a_2\to b_2\to a_3\to d\to x.
\end{equation}
If $q=3$, this is already a $C_9$.  Assume $q\ge4$.  The suffixes of~\eqref{eq:safe-chain} give $x$--$a_3$ paths
\[
  Q_3=xa_2b_2a_3,\qquad
  Q_5=xa_1b_1a_2b_2a_3
\]
of lengths three and five.  Since $a_3\notin W$, take a transition
\[
  a'\to u\to v\to a_3.
\]
Then
\[
  Q_4=xa'uva_3
\]
is a simple path of length four.  Different candidate paths may intersect; only one will eventually be used.  Let $Z$ be the union of their internal vertices.  Then $|Z|\le7$.

Since $|Z|\le7$, we have
\[
  |Z|+1+\ell-8\le\ell\le m,
  \qquad
  15(|Z|+\ell-8)+7\le15(\ell-1)+7=n_0(q).
\]
Lemma~\ref{lem:extension-deletion}(iii), with $B=8$, now closes one of
$Q_3,Q_4,Q_5$ to a $C_\ell$.

\smallskip
\noindent\emph{Case 2: a dominant row.}
Suppose that a row type $S$ occurs at least
\[
  |A_0|-1\ge m-4
\]
times.  For every $b\in S$, we have $s_b\ge m-4$, so
\begin{equation}\label{eq:heavy-row}
  e(S,B\setminus S)
  \ge m(m-4)-\binom m2
  =\frac{m^2-7m}{2}.
\end{equation}
Choose $d\in S\cap C$, which exists by~\eqref{eq:row-meets-C}, and
form the bipartite graph with parts
\[
  S\setminus\{d\}
  \quad\text{and}\quad
  (B\setminus S)\setminus\{x\},
\]
joining $b$ to $c$ precisely when $b\to c$ in $G$.  Deleting the
source $d$ and the target $x$ removes at most $2m$ arcs
from~\eqref{eq:heavy-row}, so this bipartite graph has at least
$(m^2-11m)/2$ edges.  If its maximum matching had size at most $q-2$,
K\"onig's theorem would give a vertex cover of size at most $q-2$.
Every vertex of the bipartite graph has degree at most $m$, so the cover
would meet at most $(q-2)m$ edges.  By~\eqref{eq:first-port-numerics},
\[
  \frac{m^2-11m}{2}-(q-2)m
  =\frac m2(m-2q-7)>0,
\]
a contradiction.  Take a matching
\[
  b_i\to c_i\qquad(0\le i\le q-2)
\]
and distinct $S$-type vertices $a_0,\ldots,a_{q-1}$.  These choices are
available because $m-4\ge q$ by~\eqref{eq:first-port-numerics}.  Then
\[
  x\to a_0\to b_0\to c_0\to a_1\to\cdots
  \to b_{q-2}\to c_{q-2}\to a_{q-1}\to d\to x
\]
is a $C_{3q}$.

\smallskip
\noindent\emph{Case 3: variation on at most three points.}
Suppose the active variation has size at most three.  Put
\[
  K=\bigcap_{a\in A_0}O(a),\qquad
  V=\left(\bigcup_{a\in A_0}O(a)\right)\setminus K,
  \qquad L=B\setminus(K\cup V).
\]
Write $v=|V|$ and $|K|=m-r$.  Every row is $K\cup P_a$ with $P_a\subseteq V$ and $|P_a|=r$.  If $v>0$, each point of $V$ appears in some but not all rows, so $1\le r\le v-1$.  Hence
\[
  (v,r)\in\{(0,0),(2,1),(3,1),(3,2)\}.
\]
Every $b\in K$ belongs to all rows indexed by $A_0$, and therefore $s_b\ge m-3$.  Since $|L|=m+r-v\le m$, we obtain
\begin{align}
e(K,L)
&\ge |K|(m-3)-\binom{|K|}{2}-|K|v\notag\\
&=\frac{(m-r)(m+r-5-2v)}2\notag\\
&\ge\frac{m^2-11m+10}{2}.
\label{eq:hub-count}
\end{align}
Fix $a_*\in A_0$ and $d\in O(a_*)\cap C$.  Since $x$ belongs to no
row, $x\in L$.  Form the bipartite graph whose left part is
$K\setminus\{d\}$ if $d\in K$ and is $K$ otherwise, whose right part
is $L\setminus\{x\}$, and whose edges are the arcs directed from left
to right.  Deleting the target $x$ and, when necessary, the source $d$
removes at most $2m$ arcs from~\eqref{eq:hub-count}; hence at least
$(m^2-15m+10)/2$ edges remain.  If there were no matching of size
$q-1$, K\"onig's theorem would give a vertex cover of size at most
$q-2$.  Both parts have size at most $m$, so such a cover would meet at
most $(q-2)m$ edges.  By~\eqref{eq:first-port-numerics},
\[
  m^2-(2q+11)m+10>0
\]
and hence there is a matching $b_i\to c_i$ of size $q-1$ from $K$ to
$L$, avoiding $d$ and $x$.  Since $|A_0|\ge m-3\ge q$, choose
distinct $a_0,\ldots,a_{q-1}\in A_0$ with $a_{q-1}=a_*$.  Since every
row contains $K$ and avoids $L$,
\[
  x\to a_0\to b_0\to c_0\to a_1\to\cdots
  \to b_{q-2}\to c_{q-2}\to a_{q-1}\to d\to x
\]
is again a $C_{3q}$.
\end{proof}

\section{The transitive-entry balanced cut}\label{sec:second-port}

We now consider a balanced cut $(A,B)$ rooted at $z$, together with distinct vertices $x,h,z\in B$ satisfying
\begin{equation}\label{eq:external-port}
  z\to A,\qquad x\to h\to z,\qquad x\to z.
\end{equation}
The prescribed vertex is $x$, not $z$.

\begin{proposition}\label{prop:second-port}
Let $q\ge2$, put $\ell=3q$, and suppose $n=3m\ge n_0(q)$.
Under~\eqref{eq:external-port}, the vertex $x$ lies on a copy of $C_\ell$.
\end{proposition}

\begin{proof}
Since $n=3m\ge n_0(q)$, we have
\[
  m\ge
  \begin{cases}
    13,&if~q=2,\\
    33,&if~q=3,\\
    15q-2,&if~q\ge4.
  \end{cases}
\]
In particular,
\begin{equation}\label{eq:second-port-numerics}
  m\ge\ell,\qquad m\ge2q-1,\qquad m\ge q.
\end{equation}
We distinguish whether $x$ sends an arc to $A$, whether the rows are
distinct, and whether all rows are equal.  The first two cases use three
initial path lengths and short linking.  The last case uses a bipartite
matching.

\smallskip
\noindent\emph{Case 1: an arc from $x$ to $A$.}
Suppose first that $x\to a$ for some $a\in A$.  There are $x$--$a$ paths of lengths one, two, and three:
\[
  xa,\qquad xza,\qquad xhza.
\]
With $B=6$ and $Z=\{h,z\}$, put
$R=|Z|+\ell-B=\ell-4$.  The two numerical conditions in
Lemma~\ref{lem:extension-deletion}(iii) are
\[
  \ell-3\le m
  \qquad\text{and}\qquad
  15R+7=45q-53\le n_0(q).
\]
Thus that lemma gives a $C_\ell$ through $x$.  This includes the case
$\ell=B=6$.

\smallskip
\noindent\emph{Case 2: two distinct rows.}
We may assume that $A\to x$ and that two rows $O(a),O(a')$ are
distinct.  Since $x$ belongs to every row and $z$ belongs to none, choose
\[
  b\in O(a)\setminus O(a')
\]
with $b\notin\{x,z\}$.  Then $a\to b\to a'$.  If $\ell\ge9$, the three $x$--$a'$ paths
\[
  xza',\qquad xhza',\qquad xzaba'
\]
have lengths two, three, and four.  Apply
Lemma~\ref{lem:extension-deletion}(iii) with $B=7$ and
$Z=\{h,z,a,b\}$.  Indeed,
\[
  |Z|\le4,\qquad |Z|+1+\ell-B\le\ell-2\le m,
  \qquad R=|Z|+\ell-B\le\ell-3,
\]
and, since $q\ge3$,
\[
  15R+7\le45q-38\le n_0(q).
\]

If $\ell=6$, choose the order of the two rows so that $b\ne h$.  This is always possible: if $O(a)\setminus O(a')=\{h\}$, then every element of $O(a')\setminus O(a)$ differs from $h$.  Now
\[
  x\to h\to z\to a\to b\to a'\to x
\]
is a $C_6$.

\smallskip
\noindent\emph{Case 3: all rows are equal.}
It remains that $A\to x$ and all rows are equal to one $m$-set $S$.
Then $x\in S$ and $z\notin S$.  Every $b\in S$ has $s_b=m$, and hence
\[
  e(S,B\setminus S)
  \ge m^2-\binom m2
  =\frac{m(m+1)}2.
\]
Form the bipartite graph with source part $S\setminus\{x\}$, target
part $(B\setminus S)\setminus\{z\}$, and an edge $bc$ whenever
$b\to c$.  Deleting the source $x$ removes at most $m$ arcs.  After
that deletion, the target $z$ is incident with at most $m-1$ remaining
sources.  Equivalently, the known arc $x\to z$ is not counted twice in
the two deletions.  Hence at least
\[
  \frac{(m-1)(m-2)}2
\]
edges remain.  If there were no matching of size $q-1$, K\"onig's
theorem would give a vertex cover of size at most $q-2$, which meets at
most $(q-2)(m-1)$ edges.  By~\eqref{eq:second-port-numerics},
\[
  \frac{(m-1)(m-2)}2-(q-2)(m-1)
  =\frac{m-1}{2}(m-2q+2)>0,
\]
a contradiction.  A matching $b_i\to c_i$ of size $q-1$ and distinct
vertices $a_0,\ldots,a_{q-1}\in A$, available because $m\ge q$
by~\eqref{eq:second-port-numerics}, yield
\[
  x\to z\to a_0\to b_0\to c_0\to a_1\to\cdots
  \to b_{q-2}\to c_{q-2}\to a_{q-1}\to x,
\]
a cycle of length $3q$.
\end{proof}

\section{The butterfly branch and the main theorem}\label{sec:main-proof}

The next lemma adapts~\cite[Lemma~19]{KKO}.  Its point is that, when
$3\mid n$, the butterfly argument lowers the semidegree hypothesis from
$n/3+1$ to $n/3$.

\begin{lemma}\label{lem:butterfly-C6}
Let $G$ be an oriented graph on $n$ vertices, put
$d=\lceil n/3\rceil$, and suppose $d\ge3$ and $\delta^0(G)\ge d$.
If $G$ contains an $xy$-butterfly, then $x$ lies on a copy of $C_6$.
\end{lemma}

\begin{proof}
Write the butterfly vertices as $x,h,z,b,y$, with arcs
\[
  xh,xz,hz,zb,zy,by.
\]
We first translate the absence of a $C_6$ through $x$ into three
restrictions on return paths from $y$ to $x$.  Two neighbourhood layers
then have only one possible overlap, and a cardinality estimate forces
that overlap to contain a vertex different from $z$.

The butterfly contains the $x$--$y$ paths
\[
  xhzby,\qquad xhzy,\qquad xzy
\]
of lengths four, three, and two, respectively.  Any $y$--$x$ path of
length two automatically avoids $h,z,b$: using one of these vertices as
its internal vertex would contradict, respectively, $x\to h$, $z\to y$,
or $b\to y$.  Such a path therefore closes with $xhzby$.  A length-three
$y$--$x$ path cannot contain $z$, because $z\to y$ and $x\to z$; if it
also avoids $h$, it closes with $xhzy$.  Finally, a length-four
$y$--$x$ path avoiding $z$ closes with $xzy$.  Thus, if no $C_6$
through $x$ exists, we may assume that
\begin{enumerate}
\item there is no $y$--$x$ path of length two;
\item there is no $y$--$x$ path of length three avoiding $h$;
\item there is no $y$--$x$ path of length four avoiding $z$.
\end{enumerate}

Choose
\[
  Y\subseteq N^+(y)\setminus\{h,x\},\quad |Y|=d-2,
  \qquad
  X\subseteq N^-(x)\setminus\{y\},\quad |X|=d-1.
\]
Set
\[
  Y'=N^+(Y)\setminus Y,\qquad X'=N^-(X)\setminus X.
\]
The three assumptions imply
\[
  X\cap Y=X\cap Y'=Y\cap X'=\varnothing.
\]
Indeed, a vertex in $X\cap Y$ gives a length-two $y$--$x$ path.  If
$w\in X\cap Y'$, choose $u\in Y$ with $u\to w$; then
$y\to u\to w\to x$.  If $w\in Y\cap X'$, choose $v\in X$ with
$w\to v$; then $y\to w\to v\to x$.  In both length-three paths the
internal vertices avoid $h$, since $h\notin X$ by $x\to h$ and
$h\notin Y$ by the definition of $Y$.
They also imply that $x,y$ lie in none of $X,Y,X',Y'$.  Indeed, the definitions exclude $x$ from $X,Y$ and $y$ from $X,Y$; if $x\in Y'$ or $y\in X'$, then there is a $y$--$x$ path of length two, while orientedness excludes $x\in X'$ and $y\in Y'$.  By Lemma~\ref{lem:elementary},
\[
  |Y'|\ge\left\lceil\frac{d+3}{2}\right\rceil,
  \qquad
  |X'|\ge\left\lceil\frac{d+2}{2}\right\rceil,
\]
and hence $|X'|+|Y'|\ge d+3$.  Since every overlap among the four sets is contained in $X'\cap Y'$, we obtain
\[
  n+|X'\cap Y'|
  \ge |X|+|Y|+|X'|+|Y'|+2
  \ge3d+2.
\]
Since $n\le3d$, we have $|X'\cap Y'|\ge2$.  Choose
$w\in(X'\cap Y')\setminus\{z\}$.  There are vertices $u\in Y$ and
$v\in X$ such that
\[
  y\to u\to w\to v\to x.
\]
Since $z\notin Y$ by $z\to y$, $z\notin X$ by $x\to z$, and
$w\ne z$, this is a length-four path avoiding $z$, contradicting (iii).
\end{proof}

For $q\ge3$, the three paths in a butterfly can be combined with the
short-linking lemma.

\begin{lemma}\label{lem:butterfly-long}
Let $q\ge3$, put $\ell=3q$, and let $G$ be an oriented graph on
$n\ge n_0(q)$ vertices with
$\delta^0(G)\ge\lceil n/3\rceil$.  If $G$ contains an
$xy$-butterfly, then $x$ lies on a copy of $C_\ell$.
\end{lemma}

\begin{proof}
Let $x,h,z,b,y$ be the butterfly vertices, and put
$B=7$ and $Z=\{h,z,b\}$.  The butterfly supplies $x$--$y$ paths of
lengths two, three, and four.  Moreover,
\[
  |Z|+1+\ell-B=\ell-3\le\left\lceil\frac n3\right\rceil,
  \qquad R=|Z|+\ell-B=\ell-4,
\]
and, since $q\ge3$,
\[
  15R+7=45q-53\le n_0(q).
\]
Lemma~\ref{lem:extension-deletion}(iii) gives a $C_\ell$ through $x$.
\end{proof}

We can now prove the critical case.

\begin{theorem}\label{thm:critical}
Let $q\ge2$, put $\ell=3q$, and let $G$ be an oriented graph on
$n=3m\ge n_0(q)$ vertices with $\delta^0(G)\ge m$.  Then every vertex
of $G$ lies on a copy of $C_\ell$.
\end{theorem}

\begin{proof}
Fix $x\in V(G)$.  If $N^+(x)$ is independent,
Lemma~\ref{lem:balanced-cut} gives the root-dominating balanced cut.
Proposition~\ref{prop:first-C6} applies when $q=2$, and
Proposition~\ref{prop:first-port} applies when $q\ge3$.

Suppose $N^+(x)$ is not independent.  Choose an arc $h\to z$ inside $N^+(x)$.  Thus
\[
  x\to h\to z,\qquad x\to z.
\]
If $N^+(z)$ is independent, Lemma~\ref{lem:balanced-cut} gives the transitive-entry balanced cut, and Proposition~\ref{prop:second-port} applies.

Finally, suppose $N^+(z)$ is not independent.  Choose an arc $b\to y$ inside $N^+(z)$.  The five vertices $x,h,z,b,y$ are distinct, and
\[
  xh,xz,hz,zb,zy,by
\]
form an $xy$-butterfly.  Apply Lemma~\ref{lem:butterfly-C6} for $q=2$ and Lemma~\ref{lem:butterfly-long} for $q\ge3$.
\end{proof}

\begin{proof}[Proof of Theorem~\ref{thm:main}]
Put $\ell=3q$ and fix $x\in V(G)$.  If $3\mid n$, the result is
Theorem~\ref{thm:critical}.  Suppose that $3\nmid n$ and put
$d=\lceil n/3\rceil$.  Lemma~\ref{lem:elementary}(ii) gives
\[
  |I|\le n-2d<d
\]
for every independent set $I$.  Since every outneighbourhood has size
at least $d$, none is
independent.  Choose an arc $h\to z$ inside $N^+(x)$, and then choose
an arc $b\to y$ inside $N^+(z)$.  The five vertices are distinct:
$b,y$ cannot equal $x$ or $h$ because $x\to z$ and $h\to z$.
Thus
\[
  xh,xz,hz,zb,zy,by
\]
form an $xy$-butterfly.  Lemma~\ref{lem:butterfly-C6} applies when
$q=2$, and Lemma~\ref{lem:butterfly-long} applies when $q\ge3$.
\end{proof}

The following construction of Kelly, K\"uhn and Osthus
\cite[discussion following Theorem~4]{KKO}, with relabelled parts,
proves that the semidegree bound
in Corollary~\ref{cor:threshold} is best possible.
On $n-1$ vertices, take three independent sets $V_1,V_2,V_3$ whose
sizes differ by at most one, and add all arcs
\[
  V_1\to V_2,\qquad V_2\to V_3,
  \qquad V_3\to V_1.
\]
Add a vertex $u$ with
\[
  V_1\to u\to V_2
\]
and no arcs between $u$ and $V_3$.  The resulting oriented graph satisfies
the following degree table, where $a=|V_1|$, $b=|V_2|$, and
$c=|V_3|$:
\[
\begin{array}{c|cc}
\text{vertex class}&d^+&d^-\\ \hline
V_1&b+1&c\\
V_2&c&a+1\\
V_3&a&b\\
\{u\}&b&a
\end{array}
\]
Since $a,b,c$ differ by at most one, this gives
\[
  \delta^0(G)
  =\min\{a,b,c\}
  =\left\lfloor\frac{n-1}{3}\right\rfloor
  =\left\lceil\frac n3\right\rceil-1.
\]
Every directed cycle through $u$ consists of the arc from $u$ to $V_2$, a path from $V_2$ to $V_1$ in the cyclic three-part core, and an arc from $V_1$ to $u$.  The middle path has length $2$ modulo $3$, so the whole cycle has length $1$ modulo $3$.  Hence $u$ lies on no $C_{3q}$, proving Corollary~\ref{cor:threshold}.

\begin{remark}
Proposition~\ref{prop:link-sharp} shows that the constant $+3$ in the
linking inequality cannot be reduced and that the coefficient $21$ in
Corollary~\ref{cor:defect-link} is asymptotically best possible.  This
does not determine the optimal coefficients in the order hypotheses of
Theorem~\ref{thm:main} or Corollary~\ref{cor:kko-linear}.

For $q\ge4$, an argument that uses the remaining graph only through its
order and minimum semidegree cannot lower $45q-8$ to $45q-9$.  Indeed,
put $\ell=3q$.  If $n=45q-9=15\ell-9$ and the terminal-safe-chain
branch deletes $r=\ell-1$ vertices, the remaining parameters are
\[
  N:=n-r=14\ell-8,\qquad
  d:=\frac n3-r=4\ell-2,
\]
with
\[
  \delta^0(H)\ge d,\qquad 7d=2N+2.
\]
With $c_0=2\ell-1$, the pair $(N,d)$ is the boundary pair in
Proposition~\ref{prop:link-sharp}.  Hence a smaller cutoff requires
additional information about the remaining graph, such as the
structure of the deleted path or the relation between the in- and
outdegree losses.
It remains open to determine the smallest constant $\gamma$ for which
an order hypothesis of the form $n\ge\gamma q+O(1)$ suffices in
Theorem~\ref{thm:main}.
\end{remark}


\begin{thebibliography}{99}
\footnotesize

\bibitem{BT}
J.-C.~Bermond and C.~Thomassen,
\emph{Cycles in digraphs---a survey},
J. Graph Theory \textbf{5} (1981), no.~1, 1--43.

\bibitem{CMNO}
A.~Czygrinow, T.~Molla, B.~Nagle and R.~Oursler,
\emph{On even rainbow or nontriangular directed cycles},
J. Comb. \textbf{12} (2021), no.~4, 589--662.

\bibitem{DK}
S.~Kh.~Darbinyan and I.~A.~Karapetyan,
\emph{A note on short paths in oriented graphs},
Math. Probl. Comput. Sci. \textbf{33} (2010), 35--40.

\bibitem{GV}
A.~Grzesik and J.~Volec,
\emph{Degree conditions forcing directed cycles},
Int. Math. Res. Not. IMRN \textbf{2023} (2023), no.~11, 9711--9753.

\bibitem{Hag}
R.~H\"aggkvist,
\emph{Hamilton cycles in oriented graphs},
Combin. Probab. Comput. \textbf{2} (1993), no.~1, 25--32.

\bibitem{Jackson}
B.~Jackson,
\emph{Long paths and cycles in oriented graphs},
J. Graph Theory \textbf{5} (1981), no.~2, 145--157.

\bibitem{JWS}
Y.~Ji, S.~Wu and H.~Song,
\emph{On short cycles in triangle-free oriented graphs},
Czechoslovak Math. J. \textbf{68} (2018), no.~1, 67--75.

\bibitem{KKOexact}
P.~Keevash, D.~K\"uhn and D.~Osthus,
\emph{An exact minimum degree condition for Hamilton cycles in oriented graphs},
J. Lond. Math. Soc. (2) \textbf{79} (2009), no.~1, 144--166.

\bibitem{KKOdirac}
L.~Kelly, D.~K\"uhn and D.~Osthus,
\emph{A Dirac-type result on Hamilton cycles in oriented graphs},
Combin. Probab. Comput. \textbf{17} (2008), no.~5, 689--709.

\bibitem{KKO}
L.~Kelly, D.~K\"uhn and D.~Osthus,
\emph{Cycles of given length in oriented graphs},
J. Combin. Theory Ser. B \textbf{100} (2010), no.~3, 251--264.

\bibitem{KOsurvey}
D.~K\"uhn and D.~Osthus,
\emph{A survey on Hamilton cycles in directed graphs},
European J. Combin. \textbf{33} (2012), no.~5, 750--766.

\bibitem{KOP}
D.~K\"uhn, D.~Osthus and D.~Piguet,
\emph{Embedding cycles of given length in oriented graphs},
European J. Combin. \textbf{34} (2013), no.~2, 495--501.

\bibitem{WWZ}
G.~Wang, Y.~Wang and Z.~Zhang,
\emph{Arbitrary orientations of cycles in oriented graphs},
\href{https://arxiv.org/abs/2504.09794}{arXiv:2504.09794v2 [math.CO]}, 2025.

\bibitem{ZY}
J.~Zhou and J.~Yan,
\emph{Semi-degree condition for arbitrary $H$-linked oriented graphs},
\href{https://arxiv.org/abs/2407.06675}{arXiv:2407.06675v2 [math.CO]},
2024, revised 2025.

\end{thebibliography}
\end{document}